\documentclass[12pt]{article} 
 \usepackage{dsfont}
 \usepackage{amsmath}
 \usepackage{amsthm}
 \usepackage{txfonts,bm}
 \usepackage{latexsym}
 \usepackage[colorlinks,citecolor=blue,urlcolor=blue,linkcolor=blue]{hyperref}

 \newtheorem{theo}{\small\bf Theorem}
 \newtheorem{lem}{\small\bf Lemma}[section]
 
 \newtheorem{rem}{\small\bf Remark}[section]
 \newenvironment{REM}{\begin{rem} \rm}{\end{rem}}
 \newtheorem{defi}{\small\bf Definition}[section]
 
 \newtheorem{cor}{\small\bf Corollary}

 \newcommand{\be}{\begin{equation}}
 \newcommand{\ee}{\end{equation}}
 \newcommand{\E}{\operatorname{\mathds{E}}}
 \renewcommand{\Pr}{\operatorname{\mathds{P}}}
 \newcommand{\Var}{\operatorname{Var}}

 \newcommand{\bbb}[1]{\textrm{\boldmath $ #1 $}}

 \newcommand{\RR}{\mathds{R}}

 \newcommand{\CC}{\mathds{C}}

 \newcommand{\ud}{{\rm d}}

\makeatletter
\renewcommand\@biblabel[1]{#1.}
\makeatother

 \newcommand{\Bin}{\Big(\hspace{-.6ex}
 \begin{array}{c} n \\ k \end{array}
 \hspace{-.6ex}\Big)}

 \newcommand{\BinB}{\Big(\hspace{-.6ex}
 \begin{array}{c} n \\ k-1 \end{array}
 \hspace{-.6ex}\Big)}

 \newcommand{\BinC}{\Big(\hspace{-.6ex}
 \begin{array}{c} n+1 \\ k \end{array}
 \hspace{-.6ex}\Big)}

 \title{ \Large\bf Order statistics from exchangeable
 random variables are always sufficient}

 \author{\large\bf
  Nickos Papadatos
  \vspace*{-3ex}
  }

  \date{\normalsize
  Department of Mathematics, National and Kapodistrian
  University of Athens,
  Panepistemiopolis, 157 84 Athens, Greece.
  }

\begin{document}

 \maketitle

 \thispagestyle{empty}

 \begin{abstract}
 \noindent
 Let $(X_1,\ldots,X_n)$ be an exchangeable random vector with
 distribution function $F$, and
 denote by $Y_1\leq \cdots\leq Y_n$ the corresponding order statistics.
 We show that the conditional distribution of $(X_1,\ldots,X_n)$
 given $(Y_1,\ldots,Y_n)$ does not depend on $F$.
 \end{abstract}
 {\footnotesize
 {\it MSC}:  Primary 62B05; Secondary 62G30.
 \newline
 {\it Key words and phrases}: Sufficiency;
 Exchangeable Random Variables; Symmetric Borel $\sigma$-field;
 \vspace{-.3em} 
 Order Statistics; Fisher-Neyman Factorization Criterion.}
 \vspace{1.2em}

 
 For a Borel set $B\in{\cal B}(\RR^n)$ we set
 $B^{+}:=\cup_{\bbb \pi} \{(x_{\pi(1)},\ldots,x_{\pi(n)}): \ {\bbb x}
 :=(x_1,\ldots,x_n)\in B\}$, where the union is taken over the $n!$ permutations
 $\bbb{\pi}:=(\pi(1),\ldots,\pi(n))$ of $\{1,\ldots,n\}$.
 Define the conic Borel set $\Gamma_0:=\{{\bbb x}\in \RR^n: x_1\leq \cdots\leq x_n\}$.
 If ${\bbb X}:=(X_1,\ldots X_n)$ is a random vector with order statistics 
 $Y_1\leq \cdots\leq Y_n$
 then, setting  ${\bbb Y}:=(Y_1,\ldots Y_n)$, it is readily
 checked that ${\bbb Y}\in B$ if and only if (iff) ${\bbb Y}\in B\cap \Gamma_0$
 iff ${\bbb X}\in (B\cap \Gamma_0)^+$ for any Borel subset $B$ of $\RR^n$.
 This verifies that ${\cal B}^+:=\{(B\cap \Gamma_0)^+: B\in{\cal B}(R^n)\}$
 is a $\sigma$-field in $\RR^n$ which we shall call 
 {\it the symmetric $\sigma$-field}.
 A Borel subset $B$ of $\RR^n$ is called {\it symmetric} if $\bbb{x}\in B$
 implies $(x_{\pi(1)},\ldots,x_{\pi(n)})\in B$ for any permutation
 $\bbb \pi$. Under this definition, 
 ${\cal B}^+$ is the smallest $\sigma$-field that contains the symmetric
 Borel subsets of $\RR^n$. 
 
 Using the preceding terminology we have the following

 \begin{theo}
   \label{theo.1}
   If $X_1,\ldots,X_n$ are exchangeable random variables and 
   $g:\RR^n\to\RR$ is Borel measurable with $\E |g({\bbb X})|<\infty$
   then
   \[
   \E[g({\bbb X}) \ | \ {\bbb Y}] =\frac{1}{n!} \sum_{{\bbb \pi}}
   g(Y_{\pi(1)},\ldots,Y_{\pi(n)}), \ \ \mbox{\rm a.s.}
   \]
 \end{theo}
 \begin{proof}
 For every ${\bbb x}\in\RR^n$ we fix a permutation ${\bbb \pi}={\bbb \pi}_{\bbb x}$ 
 such that
 $x_{\pi(1)}\leq \cdots\leq x_{\pi(n)}$.
 Define ${\bbb f}({\bbb x})=(f_1({\bbb x}),\ldots,f_n({\bbb x})):\RR^n\to\RR^n$ where 
 $f_1({\bbb x})=y_1,\ldots,f_n({\bbb x})=y_n$, with 
 $y_i=x_{\pi(i)}$, $i=1,\ldots,n$. The mapping $\bbb f:\RR^n\to \RR^n$
 is continuous, hence Borel. Then, the $\sigma$-field generated by $\bbb Y$,
 $\sigma ({\bbb Y})$, equals to 
 \[
 \sigma ({\bbb Y})=({\bbb f}\circ {\bbb X})^{-1}
 ({\cal B}(\RR^n))
 =
 {\bbb X}^{-1}\Big({\bbb f}^{-1}({\cal B}(\RR^n)) \Big)
 ={\bbb X}^{-1} ({\cal B}^+).
 \]
 Let $(\Omega,{\cal A},\Pr)$ be the probability space where
 the random vector $\bbb X$ is defined. Clearly, 
 $A\in \sigma({\bbb Y})$ iff $A=\{\omega:{\bbb Y}(\omega)\in B\}
 =\{\omega:{\bbb X}(\omega)\in (B\cap \Gamma_0)^+\}$ for some
 $B\in {\cal B}(\RR^n)$. Thus, we may write
 $I_A(\omega)=I_{(B\cap\Gamma_0)^+}({\bbb X}(\omega))$, where
 $I$ denotes an indicator function. The result
 will be proved if we verify the identity
 \be
 \label{1}
 \int_A g({\bbb X}) \ \ud\Pr = \frac{1}{n!} \sum_{\bbb \pi} 
 \int_A g(Y_{\pi(1)},\ldots,Y_{\pi(n)}) \ \ud \Pr, 
 \ \ \ \ A\in\sigma({\bbb{Y}}).
 \ee
 Obviously,
 \be
 \label{2}
 \frac{1}{n!} \sum_{\bbb \pi}
 g(Y_{\pi(1)},\ldots,Y_{\pi(n)})=
 \frac{1}{n!} \sum_{\bbb \pi}
 g(X_{\pi(1)},\ldots,X_{\pi(n)}),
 \ee
 because these sums contain the same summands. Since $I_A=I_{(B\cap \Gamma_0)^+}(\bbb X)$,
 we can rewrite the left-hand side of (\ref{1}) as
 \[
 \int I_A \ g({\bbb X}) \ \ud\Pr =
 \int I_{(B\cap \Gamma_0)^+}(\bbb X) \ g({\bbb X}) \ \ud\Pr
 =\E h(\bbb X),
 \]
 where $h(\bbb x):=I_{(B\cap \Gamma_0)^+} (\bbb x) g(\bbb x)$.
 From (\ref{2}), the right-hand side of (\ref{1}) equals to 
 \[
 \frac{1}{n!} \sum_{\bbb \pi} \int I_{(B\cap \Gamma_0)^+} (\bbb X) \
 g(X_{\pi(1)},\ldots,X_{\pi(n)}) \ \ud \Pr.
 \]
 Since the set $(B\cap \Gamma_0)^+$ is symmetric,
 it holds $I_{(B\cap \Gamma_0)^+} (\bbb X)
 =I_{(B\cap \Gamma_0)^+} (X_{\pi(1)},\ldots,X_{\pi(n)})$ for every $\bbb \pi$.
 Using exchangeability we obtain
 \[
 \frac{1}{n!} \sum_{\bbb \pi} \int I_{(B\cap \Gamma_0)^+} (\bbb X) \
 g(X_{\pi(1)},\ldots,X_{\pi(n)}) \ \ud \Pr = 
 \frac{1}{n!} \sum_{\bbb \pi} \E h(X_{\pi(1)},\ldots,X_{\pi(n)})=\E h(\bbb X),
 \]
 and the proof is complete.
 \end{proof}

 \begin{cor}
 \label{cor.1}
 The conditional distribution of ${\bbb X}$ given $\bbb Y$ is given by
 \[
 \Pr({\bbb X}\in B \ | \ \bbb Y)=\frac{1}{n!} \sum_{\bbb \pi} 
 I_B(Y_{\pi(1)},\ldots,Y_{\pi(n)}), 
 \ \ \ \  B\in{{\cal B}(\RR^n)}.
 \]
 \end{cor}
 \begin{proof}
 Apply Theorem \ref{theo.1} with $g=I_B$.
 \end{proof}
 
 Corollary \ref{cor.1} verifies that the order statistics are {\it sufficient}
 within the entire family,  ${\cal F}$, 
 of exchangeable probability measures on $\RR^n$. This result cannot be
 proved by means of the well-known Fisher-Neyman factorization criterion,
 since there is no $\sigma$-finite  measure dominating ${\cal F}$; see
 Shao (2003), Lemma 2.1, Theorem 2.2 and Exercise 31. 
 Moreover, let
 ${\cal F}'\subset {\cal F}$ be the family 
 of $n$-variate probability measures
 of independent, identically distributed, random variables $X_1,\ldots,X_n$.
 Clearly, ${\cal F}'$ contains the Dirac measures $\delta_{(x,\ldots,x)}$, $x\in\RR$,
 and therefore, there is no $\sigma$-finite measure dominating ${\cal F}'$, too.

 \end{document}